\documentclass{article}

\usepackage{geometry}
\usepackage{graphicx,color}
\usepackage{tikz}
\usetikzlibrary{arrows.meta}
\usetikzlibrary{positioning}
\usepackage{amsmath}
\usepackage{amssymb}
\usepackage{amsthm}
\usepackage{enumerate}
\usepackage{mathtools}
\usepackage{subcaption}
\usepackage{cases}
\usepackage{multicol}
\usepackage{expl3}
\usepackage{hyperref}
\usepackage{wrapfig}

\usepackage[style=nejm, 
citestyle=numeric-comp,
sorting=none,
backend=biber]{biblatex}
\theoremstyle{plain}
\newtheorem{thm}{Theorem}[section]
\newtheorem{lemma}[thm]{Lemma}
\newtheorem{defn}[thm]{Definition}
\newtheorem{prop}[thm]{Proposition}

\newtheorem{assumption}[thm]{Assumption}

\newtheorem{remark}[thm]{Remark}

\numberwithin{equation}{section}

\newcommand{\R}{\mathbb{R}}

\newcommand{\N}{\mathbb{N}}
\newcommand{\K}{\mathcal{K}}
\newcommand{\LL}{\mathcal{L}}
\newcommand{\D}{\mathcal D}
\newcommand{\PP}{\mathcal{P}}
\newcommand{\RR}{\mathcal{R}}
\newcommand{\A}{\mathcal{A}}
\newcommand{\BB}{\mathcal{B}}
\newcommand{\CC}{\mathcal{C}}
\newcommand{\HH}{\mathcal{H}}
\renewcommand{\Re}{\mathrm{Re}}

\title{Input-to-state stability of second-order port-Hamiltonian systems under nonlinear dynamic boundary feedback}

\author{Bouchra Elghazi\thanks{University of Wuppertal, School of Mathematics and Natural Sciences, Gaußstraße 20, 42119 Wuppertal, Germany, \href{mailto:elghazi@uni-wuppertal}{elghazi@uni-wuppertal.de}.}
\and Birgit Jacob\thanks{University of Wuppertal, School of Mathematics and Natural Sciences, Gaußstraße 20, 42119 Wuppertal, Germany, \href{mailto:bjacob@uni-wuppertal}{bjacob@uni-wuppertal.de}.}
\and Christophe Prieur\thanks{Univ. Grenoble Alpes, CNRS, Grenoble INP, GIPSA-lab, France, \href{mailto:Christophe.Prieur@gipsa-lab.fr}{Christophe.Prieur@gipsa-lab.fr}.}}

\date{}

\begin{document}

\maketitle

\begin{abstract}
The input-to-state stability of a class of infinite-dimensional second-order port-Hamiltonian systems on a one-dimensional spatial domain is analyzed under nonlinear dynamic boundary feedback and boundary disturbances. Using energy methods in the port-Hamiltonian framework, we obtain sufficient conditions that guarantee uniform input-to-state stability of the closed-loop system.
\end{abstract}


\section{Introduction}\label{intro}

Port-Hamiltonian systems form a class of dynamical systems that describe the flow and storage of energy via a Hamiltonian function and an underlying interconnection structure. When extended to distributed-parameter systems, one obtains boundary port-Hamiltonian systems, where the dynamics are governed by partial differential equations and the control inputs act at the boundary of the spatial domain. The port-Hamiltonian framework has become a leading approach for modelling and controlling complex physical systems \cite{SchaftCervera:07,Schaft:06,SchJelt:14,SchaftMaschke:02}. For a comprehensive survey of port-Hamiltonian systems, we refer to \cite{RashardScjaftStram:20}.
We focus on a class of second-order linear boundary-controlled port-Hamiltonian systems on a one-dimensional spatial domain, given by
 
\begin{equation*}
    \frac{\partial x}{\partial t}(\zeta,t) = \left( P_2 \frac{\partial^2}{\partial \zeta^2} + P_1 \frac{\partial}{\partial \zeta} + P_0 \right) \mathcal{H}(\zeta) x(\zeta,t), \quad t\ge0, \quad \zeta \in (0,1) ,
\end{equation*}
subject to the boundary conditions specified in Section \ref{section2}. For this class of port-Hamiltonian systems, fundamental system-theoretic properties such as well-posedness \cite{GorZwaMas:2005,Villegas:07,ElghJacZw:25}, observer design \cite{LeGorrecToledoRamirezWu:23}, and robust output regulation \cite{HumLassi:18,humaloja2016robust,PaunLeGorrecRamirez:21,HumLassiPohj:16} have been extensively studied.
\medskip

The stability of boundary-controlled port-Hamiltonian systems is naturally analysed using energy-based methods, since the Hamiltonian represents the total system energy. In this spirit, asymptotic and exponential stability via (static and dynamic) linear and nonlinear boundary feedback has been established in \cite{Aug:16,Aug:19,AugJac:14,Aug:20,RamirezZwartLeGorrec:17,SchmidZwart:21,MattioniWuLeGorrZw:22,RamirezLeGorrMachelliZW:14,Macchelli2004,macchelli2016synthesis,VillegasZwLeGorrMasche:09}. These stability notions are fundamental, but they do not fully capture robustness in the presence of external disturbances such as noise, unmodelled dynamics, or input perturbations. 
\medskip

In many physical applications, the decisive question is whether stability can be preserved under such disturbances.
To address this issue, input-to-state stability (ISS), introduced by Sontag in the late 1980s \cite{sontag1989}, has become a key concept. Informally, ISS guarantees that the system state remains bounded in proportion to the magnitude of the disturbances, and thus that external inputs cannot drive the energy to infinity. ISS provides a natural and powerful robustness notion for infinite-dimensional systems, and has been the subject of a rich literature; see, for example, \cite{Sontag1995,Mironchenko2023,Krstic2019,MironchenkoPrieur:20,MazencPrieur:2011,MironKarafKrstic:19}.
\medskip

In this contribution, we investigate the stabilization of second-order port-Hamiltonian systems via dynamic boundary control and analyse the ISS properties of the resulting closed-loop system. More precisely, we consider the interconnection of a second-order boundary port-Hamiltonian system with a class of nonlinear boundary controllers and study input-to-state stability with respect to boundary disturbances.
The class of nonlinear controllers under consideration includes mechanical and electrical systems with nonlinear potential energy. Using such controllers, \cite{Macchelli2004} established Lyapunov stability for a Timoshenko beam with a tip mass at one end and a controller at the other end. Building on this idea, \cite{RamirezZwartLeGorrec:17,SchmidZwart:21} analysed a more general class of first-order port-Hamiltonian systems (including, in particular, the wave equation and the Timoshenko beam equation) and obtained asymptotic, exponential, and input-to-state stability results with the same controller class. For second-order port-Hamiltonian systems, which encompass Euler–Bernoulli beam models, only an intermediate stability property has been derived so far \cite{SchmidZwart:21}. 
The present work closes this gap by focusing on input-to-state stability for the interconnection of a second-order port-Hamiltonian system with a nonlinear dynamic boundary controller. We derive energy-based sufficient conditions within the port-Hamiltonian framework that ensure uniform ISS of the closed-loop system with respect to boundary disturbances.

\medskip
The paper is organized as follows. In the next section, we formulate the problem and introduce the closed-loop system under consideration. Section \ref{section3} is devoted to the analysis of the well-posedness of the closed-loop system. In Section \ref{section4}, we establish the main input-to-state stability result. An application to the beam equation is presented in Section \ref{section5}. Section \ref{section6} summarizes the results and discusses future work.

\medskip
\noindent\textbf{Notation.} As usual, $\K$, $\K_{\infty}$, $\LL$ and $\mathcal{KL}$ denote the following classes of comparison functions:
\begin{align*}
    \K &:= \left\{ \gamma \in C(\R^{+}) : \gamma \text{ strictly increasing with } \gamma(0) = 0 \right\}, \\
    \K_{\infty} &:= \left\{ \gamma \in  \K : \gamma \text{ unbounded} \right\}, \\
    \LL &:= \left\{ \gamma \in C(\R^{+}) : \gamma \text{ strictly decreasing with } \lim_{r \to \infty} \gamma(r) = 0 \right\}, \\
    \mathcal{KL} &:= \left\{ \beta \in C(\R^{+} \times \R^{+}; \R^{+}) : \beta(\cdot,t) \in \K, ~\forall t\ge 0, ~\beta(r,\cdot) \in \LL, ~\forall r>0 \right\},
\end{align*}
where $\R^{+}:=[0,\infty)$ and $\N_{0}= \N \setminus \{0\}$.

\section{Problem formulation}\label{section2}
We consider port-Hamiltonian  systems of the following form \cite{Villegas:07}
\begin{equation}\label{syst:PHS}
\begin{aligned}
     \dot{x}(\zeta,t) &= \left( P_2 \frac{\partial^2}{\partial \zeta^2} + P_1 \frac{\partial}{\partial \zeta} + P_0 \right) \mathcal{H}(\zeta) x(\zeta,t), \quad t >0, \quad \zeta \in [0,1] \\
     u(t) &= W_{B,1} \tau(\HH x)(t), \quad  
     0 = W_{B,2} \tau(\HH x)(t), \quad t>0 \\
    y(t) &= W_{C} \tau(\HH x)(t), \quad t>0  \\
    x(\zeta,0) &= x_0(\zeta), \quad \zeta \in (0,1)
\end{aligned}
\end{equation}
where $\tau : \mathrm H^2((0,1);\mathbb R^n) \to \mathbb R^{4n}$ is the trace operator given by
$$ \tau(x) = \begin{pmatrix}
    x(1) & x'(1) & x(0) & x'(0)
\end{pmatrix}^\top . $$
Here $t\ge 0$ denotes the time, $x(\cdot,t)\in \mathrm L^2((0,1),\mathbb R^n)$ is the state, $u(t)\in \mathbb R^m$ the input, $y(t)\in \mathbb R^m$ the output at time $t$. The derivative w.r.t.\ time $t$ is denoted by $\dot{x}$, while $x'$ refers to the derivative of $x$ w.r.t.\ space $\zeta$. 
We assume that $P_{0}, P_{1}, P_{2} \in \R^{n \times n}$ with $P_2$ skew-symmetric and invertible, $P_1$ symmetric and $\Re~ P_{0} \le 0$ .
Moreover, the Hamiltonian density function $\mathcal H \in \mathrm W^{1,\infty}((0,1); \R^{n\times n})$ is assumed to be  self-adjoint for all $\zeta \in [0,1]$ and there exist $m_{-}, m_{+} > 0$ such that for all $\zeta \in [0,1]$
$$ m_{-} I \le \mathcal H(\zeta) \le m_{+} I. $$ 
Note that $\mathcal H \in \mathrm W^{1,\infty}((0,1); \R^{n\times n})$ implies that the mapping  $\zeta \mapsto \mathcal H(\zeta)$ is absolutely continuous.
Furthermore, the matrices $W_{B,1}, W_C \in \R^{m \times 4n}$, where $0<m\le 2n$, and $W_{B,2}\in \R^{(2n-m) \times 4n}$ satisfy that $\left[\begin{smallmatrix} W_{B,1}\\ W_{B,2} \\ W_C \end{smallmatrix}\right]$ has full row rank.

As the state space of \eqref{syst:PHS}, we choose $X := \mathrm L^2((0,1), \R^{n})$ endowed with the norm $\|\cdot \|_{X}$ induced by the system's energy
\begin{equation*}
   \frac{1}{2} \| x\|_{X}^{2} := E(x) := \frac{1}{2} \int_{0}^{1} x(\zeta)^\top \HH(\zeta) x(\zeta) d\zeta .
\end{equation*}
We note that the norm $\|\cdot \|_{X}$ is equivalent to the standard norm of $\mathrm L^{2}((0,1), \R^{n})$ and it is induced by a scalar product which we denote by $\langle \cdot, \cdot \rangle_{X}$. We define the operators $ \mathcal{A}:  D(\mathcal{A})\subset X\rightarrow X$ and  $ \mathcal{B}, \mathcal{C}:  D(\mathcal{A})\subset X\rightarrow \mathbb R^{m}$ by 
\begin{align*}
    \mathcal{A} x &= \left( P_2 \frac{\partial^2}{\partial\zeta^2}  + P_1 \frac{\partial}{\partial \zeta} + P_0 \right) \mathcal H x, \\
    \BB x &= W_{B,1} \tau(\mathcal H x), \\
    \CC x &= W_{C}  \tau(\mathcal H x), \\
    D(\mathcal{A}) &= \left\{ x \in \mathrm L^2((0,1);\mathbb R^n) \, | \, \mathcal H x \in \mathrm H^2((0,1);\mathbb R^n), \, W_{B,2}  \tau(\mathcal H x)=0 \right\}. 
\end{align*}
Thus, system \eqref{syst:PHS} has the following compact representation
\begin{align*}
    \dot{x}(t) &= \mathcal A x(t),\qquad x(0)=x_0,\qquad 
    u(t)= \mathcal B x(t),\qquad
    y(t)= \mathcal C x(t), \quad t> 0.
\end{align*}
In the following, we always assume that the impedance passivity inequality
\begin{equation}\label{eqn:impedance_passive}
    \left\langle x, \A x \right\rangle_{X} \le \left(\BB x\right)^\top \CC x , \quad x \in D(\A) .
\end{equation}
holds.

If system \eqref{syst:PHS} satisfies  \eqref{eqn:impedance_passive}, then the system is called \emph{impedance passive}. If equality holds, 
then system \eqref{syst:PHS} is called \emph{impedance energy-preserving}.

\begin{remark}
The impedance passivity inequality \eqref{eqn:impedance_passive} implies that the operator $A:= \mathcal A_{|D(\A) \cap \ker \BB}$ generates a contraction semigroup \cite[Theorem 4.1]{GorZwaMas:2005} and that the port-Hamiltonian system \eqref{syst:PHS} is a boundary control system. Furthermore, for $u \in C^{2}([0,\infty);\R^{m})$ with $x \in D(\A)$ and $u(0)=\BB x_{0}$, the system \eqref{syst:PHS} possesses a unique classical solution.
If in addition, $m=2n$ and $\Re P_{0}=0$, then the port-Hamiltonian system \eqref{syst:PHS} satisfies the balance equation \cite[Theorem 6.24]{Villegas:07}
\begin{equation*}
    \left\langle \A x, x\right\rangle = \frac{1}{2} \left\langle \begin{pmatrix}
        \BB x \\ 
        \CC x 
    \end{pmatrix} , P_{W_B, W_C} \begin{pmatrix}
        \BB x \\ 
        \CC x 
    \end{pmatrix} \right\rangle .
\end{equation*}
Thus, the inequality \eqref{eqn:impedance_passive} is equivalent to 
\begin{equation*}
P_{W_B, W_C} =
\begin{bmatrix}
    \tilde W_{B} \Sigma \tilde W_{B}^\top & \tilde W_{C} \Sigma \tilde W_{B}^\top \\
    \tilde W_{C} \Sigma \tilde W_{B}^\top & \tilde W_{C} \Sigma \tilde W_{C}^\top
\end{bmatrix}^{-1}
\le \begin{bmatrix}
    0 & I \\
    I & 0
\end{bmatrix} ,
\end{equation*}
where 
\begin{equation*}
     \tilde W_B = \frac{1}{\sqrt{2}} \left[\begin{matrix}
         W_{B,1} \\ W_{B,2}
     \end{matrix}\right] \left[\begin{matrix}
           R & I \\
          -R & I
     \end{matrix} \right], \qquad \tilde W_C= \frac{1}{\sqrt{2}} W_C \left[\begin{matrix}
           R & I \\
          -R & I
     \end{matrix} \right] , \qquad
     \Sigma =\left[\begin{matrix}  0 & I \\ I & 0 \end{matrix}\right], \qquad 
        R= \left[\begin{matrix}
         0 & -P_2^{-1} \\
         P_2^{-1} & P_2^{-1} P_1 P_2^{-1}
     \end{matrix} \right] .
\end{equation*}
\end{remark}

In order to  stabilize the port-Hamiltonian system \eqref{syst:PHS}, we choose a finite-dimensional nonlinear controller given by 
\begin{equation}\label{syst:controller}
    \begin{aligned}
        \dot{v}(t) &= \begin{pmatrix}
            \dot v_{1}(t) \\
            \dot v_{2}(t)
        \end{pmatrix} = \begin{pmatrix}
            K v_{2}(t) \\
            -\nabla \PP(v_{1}(t)) - \RR(K v_{2}(t)) + B_{c} u_{c}(t)
        \end{pmatrix} , \quad t>0, \qquad v(0)=v_0,\\
        y_{c}(t) &= B_{c}^\top K v_{2}(t) + S_{c} u_{c}(t) , \quad t>0,
    \end{aligned}
\end{equation}
where the matrix $B_{c} \in \R^{n_{c}\times m}$ is given and the matrices $K \in \R^{n_c \times n_c}$, $S_{c} \in \R^{m \times m}$ are given and positive definite. Therefore, we have 
\begin{equation}\label{eqn:S_c} 
    S_{c} \ge  \delta I >0, \quad \text{where } \delta:= \min \sigma(S_{c}). 
\end{equation}
The controller state space is $V := \R^{2 n_c}$ and we choose the norm $|\cdot|_{V}$ defined by
\begin{equation*}
    |v|_{V}^{2} = \left |\begin{pmatrix}
        v_{1} \\ v_{2}
    \end{pmatrix}\right |^2 = |v_{1}|^2 + v_{2}^\top K v_{2} .
\end{equation*}
\begin{remark}
 The energy associated with this system is given by
\begin{equation}\label{eqn:energy_c}
    E_{c}(v) := \PP(v_{1}) + \frac{1}{2} v_{2}^\top K v_{2} .
\end{equation}
\end{remark}
Furthermore, we make the following assumptions on the controller \eqref{syst:controller} throughout this paper:
\begin{enumerate}
    \item The potential energy $\PP : \R^{n_{c}} \to [0, \infty)$ is differentiable such that $\nabla \PP$ is locally Lipschitz continuous with $\PP(0) =0$ and  $\PP$ is positive definite. Furthermore, $\PP(v_1) \to \infty$ as $|v_1| \to \infty$.
    \item The damping $\RR : \R^{n_c} \to \R^{n_c}$ is locally Lipschitz continuous with $\RR(0) = 0$ and nonnegative, that is, $v_2^\top \RR (v_2) \ge 0$ for all $v_2 \in \R^{n_c}$.
\end{enumerate}

These assumptions  imply, in particular, that system \eqref{syst:controller} possesses a local solution for every initial condition. Further,
the second item  combined with \eqref{eqn:S_c} implies that the controller system \eqref{syst:controller} is strictly input-passive w.r.t. the storage function $E_{c}$, that is, along solutions  
    $$ \dot{E}_{c}(v) \le u_{c}^\top y_{c} - \delta |u_{c}|^{2} . $$
Thus, the solutions cannot blow up in finite time and therefore exist globally. 
Further,  system \eqref{syst:controller} is a port-Hamiltonian system. 

The finite-dimensional passive system \eqref{syst:controller} can be interpreted from two complementary viewpoints. It may represent an intrinsic part of the physical system, such as a tip mass, a mass-spring-damper system or an attached oscillator. Alternatively, it can be regarded as an external control device designed to modify the behavior of the distributed system, for instance as a vibration absorber, a passive controller, or an energy-dissipating actuator. 

\subsection{The closed-loop system}
We consider the interconnection between the infinite-dimensional port-Hamiltonian system \eqref{syst:PHS} and the finite-dimensional nonlinear port-Hamiltonian system \eqref{syst:controller} given by
\begin{equation*}
   u_{c}(t) = y(t) + d_{c}(t) \quad  \text{and} \quad u(t) = - y_c(t) + d(t),\qquad t>0,
\end{equation*}
\begin{wrapfigure}{r}{0.4\textwidth}
    \centering
    \begin{tikzpicture}[scale=1.4]
    \node[draw,minimum width=1.8cm, minimum height=1.3cm] (A) {\shortstack{Port-Hamiltonian \\\\ system}};
    \node[draw,minimum width=1.8cm, minimum height=1.3cm, below=0.5cm of A] (B) {Controller};
    
    \draw[->,>=latex] ([xshift=-0.8cm,yshift=0.08cm]A.west) -- ([yshift=0.08cm]A.west);
    \draw ([xshift=-0.5cm,yshift=0.08cm]A.west) node[above]{$u$};
    \fill ([xshift=-0.8cm,yshift=0.08cm]A.west) circle (0.7pt);
    
    \draw[->,>=latex] ([xshift=-1.3cm,yshift=0.08cm]A.west) -- ([xshift=-0.8cm,yshift=0.08cm]A.west);
    \draw ([xshift=-1.1cm,yshift=0.1cm]A.west) node[above]{$d$};
    
    \draw ([xshift=-0.78cm,yshift=-0.04cm]A.west) node[left] {$-$};
    
    \draw[->,>=latex] ([yshift=0.08cm]A.east) --([xshift=0.91cm,yshift=0.08cm]A.east);
    \draw ([xshift=0.5cm,yshift=0.09cm]A.east) node[above] {$y$};
    \draw[<-,>=latex] ([xshift=-0.8cm,yshift=0.08cm]A.west) |-([yshift=0cm]B.west);
    \draw ([xshift=-0.5cm,yshift=0.03cm]B.west) node[above] {$y_{c}$};
    
    \draw[->,>=latex] ([xshift=0.9cm,yshift=0.08cm]A.east) |-([xshift=0cm,yshift=0cm]B.east);
    
    \draw ([xshift=0.5cm,yshift=0.03cm]B.east) node[above] {$u_{c}$};
    
    \fill ([xshift=1.22cm,yshift=0.003cm]B.east) circle (0.7pt);
    
    \draw[<-,>=latex] ([xshift=1.22cm,yshift=0cm]B.east) -- ([xshift=1.72cm,yshift=0cm]B.east);

    \draw ([xshift=1.5cm,yshift=0.03cm]B.east) node[above]{$d_{c}$};
\end{tikzpicture}
    \caption{The closed-loop system.}
    \label{fig:system}
\end{wrapfigure}
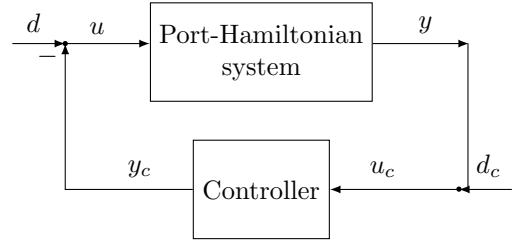
that is, we interconnect these systems in an antisymmetric way which guarantees that the interconnected system is a port-Hamiltonian system as well. The state space for the closed-loop system is the Hilbert space $\tilde X := X \times V = X \times \R^{2n_{c}}$ with norm defined by
$$ \|\tilde x\|^{2} = \|(x,v)\|^{2} := \|x\|^{2}_{X} + |v|^{2}_{V} . $$
We define the linear operators $\tilde \A: D(\tilde \A) \to \tilde X$ and $\tilde \BB, \tilde \CC : D(\tilde \A) \to \R^{m}$, the nonlinear operator $\tilde f: \tilde X \to \tilde X$, with $D(\tilde \A)=D(\A) \times \R^{2n_{c}}$ and  $\tilde B_{c}: \R^{n_{c}} \to \tilde X$ as  
\begin{align*}
    \tilde \A \tilde x &= \begin{pmatrix}
        \A x \\
        K v_{2} \\
        -v_{1} + B_{c} \CC x
    \end{pmatrix}, \qquad \tilde f(\tilde x) = \begin{pmatrix}
        0 \\
        0 \\
        v_{1} - \nabla \PP(v_{1}) - \RR(K v_{2})
    \end{pmatrix}, \\
    \tilde \BB \tilde x &:= \BB x + B_{c}^\top K v_{2} + S_{c} \CC x, \qquad \tilde B_{c} v_{2} = \begin{pmatrix}
        0 \\
        0 \\
        v_{2}
    \end{pmatrix}, \\ 
    \tilde \CC \tilde x &:= \CC x .
\end{align*}
Using the above definitions, the closed-loop system with inputs $d$ and $d_c$, and output $y$ can be written as a semilinear boundary control system, as in \cite{Mironchenko:24}
\begin{equation}\label{syst:closed-loop}
\begin{aligned}
    \dot{\tilde{x}}(t) &= \tilde \A \tilde x(t) + \tilde f(\tilde x(t)) + \tilde B_{c} B_{c} d_{c}(t), \qquad t>0, \quad \tilde x(0)=\tilde x_0,\\
    \tilde \BB \tilde x(t) &= d(t) - S_{c} d_{c}(t), \qquad t>0, \\
    \tilde \CC \tilde x(t) &= y(t), \qquad t>0,
\end{aligned}
\end{equation}

\begin{remark}
The energy of the closed-loop system in the state $\tilde x=(x,v) \in \tilde X$ is
$$ \tilde E(\tilde x) := E(x) + E_{c}(v) = \frac{1}{2} \|x\|^{2}_{X} + \PP(v_{1}) + \frac{1}{2} v_{2}^\top K v_{2}. $$
By virtue of Lemma 2.5 in \cite{Clarke:98}, there exist $\phi_{1}$, $\phi_{2} \in \K_{\infty}$ such that for all $\tilde x \in \tilde X$
\begin{equation}\label{eqn:energy_equivalence}
    \phi_{1}(\|\tilde x\|) \le \tilde E(\tilde x) \le \phi_{2}(\| \tilde x\|). 
\end{equation}
\end{remark}

\section{Well-posedness of the closed-loop system}\label{section3}
In this section, we investigate the well-posedness of the closed-loop system \eqref{syst:closed-loop}. More precisely, we establish the existence and uniqueness of solutions for sufficiently regular initial states $\tilde x_0$ and disturbances $d, d_{c}$. We first establish local classical solvability for compatible initial states and sufficiently regular disturbances. An energy estimate then yields global existence. Finally, a continuous-dependence estimate extends the solution map to arbitrary initial states in $\tilde X$ and disturbances in $\mathrm L_{\mathrm{loc}}^{2}$.

\begin{defn}\cite[Definition 4.2]{TuscnakWeiss:14}
A triple ($\tilde x, (d, d_{c}),y)$ is called a \emph{classical solution} of the semilinear boundary control system \eqref{syst:closed-loop} on $[0,t_{0})$ if $d, d_{c} \in C^{2}([0,t_{0});\R^{m})$, $y \in C([0,t_{0}); \R^{m})$, $ \tilde x \in C([0,t_{0});D(\tilde \A)) \cap C^{1}([0,t_{0}); \tilde X) $ 
and equations \eqref{syst:closed-loop} are satisfied for every $t \in (0,t_{0})$. \\
The triple $(\tilde x, (d, d_{c}),y)$ is called a \emph{generalized solution} of \eqref{syst:closed-loop} on $[0,t_{0})$ if $\tilde x \in C([0,t_{0}); \tilde X)$, $d,d_{c},y \in \mathrm L^{2}_{\mathrm{loc}}((0,t_{0});\R^{m}) $ and there exists a sequence $(\tilde x_{n}, (d_{n}, d_{c,n}), y_{n})$ of classical solutions to \eqref{syst:closed-loop} such that  
    \begin{align*}
        \tilde x_{n} &\to \tilde x \text{ in } C([0,t_{0});\tilde X)\\
        d_{n} &\to d \text{ in } \mathrm L^{2}_{\mathrm{loc}}([0,t_{0});\R^{m}), \quad  d_{c,n} \to d_c \text{ in } \mathrm L^{2}_{\mathrm{loc}}([0,t_{0});\R^{m}), \\
        y_{n} &\to y \text{ in } \mathrm L^{2}_{\mathrm{loc}}([0,t_{0});\R^{m}) .
    \end{align*}
\end{defn}

The following result follows from Lemma 2.3 in \cite{SchmidZwart:21}. Although our formulation includes an additional disturbance $d_{c}$, the proof follows analogously since the impedance passivity inequality, which is the key argument in the proof, still holds. Therefore, we have the following lemma.
\begin{lemma}
 The operator $\tilde A := \tilde \A_{|D(\tilde \A)\cap \ker \tilde \BB}$ is the infinitesimal generator of a contraction semigroup on $\tilde X$.
\end{lemma}

Given that $\begin{bsmallmatrix}
    W_{B,1} \\
    W_{B,2} \\
    W_{C}
\end{bsmallmatrix}$ has full row rank, the operator $\tilde \BB$ is surjective. It follows then that there exists a linear map $R: \R^{m} \to D(\tilde \A)$ such that $\tilde \BB R = I_{m}$. Using the classical transformation (see \cite{Fattorini:68})
\begin{equation*}
    z(t) = \tilde x(t) - R d(t) + R S_{c} d_{c}(t) ,\qquad t\ge 0, 
\end{equation*}
we can rewrite the closed-loop system as the following semilinear evolution equation (with a  time-dependent nonlinearity) 
\begin{equation}\label{eqn:semilinear_eqn}
\begin{aligned}
    \dot{z}(t) &= \tilde A z(t) + \tilde f(z(t) + R d(t) - R S_{c} d_{c}(t)) + \left( \tilde B_{c} B_{c} - \tilde \A R S_{c}\right) d_{c}(t) + R S_{c} \dot{d}_{c}(t) + \tilde \A R d(t) - R \dot{d}(t), \, t>0, \\
    y(t) &= \tilde \CC \left(z(t) + R d(t) -R S_{c} d_{c}(t)\right), \quad t>0,
\end{aligned}
\end{equation}
with initial condition $z(0)=z_0$.
Since $\tilde A$ is the infinitesimal generator of a $C_{0}$-semigroup and $\tilde f$ is locally Lipschitz continuous, 
we see that for $z_{0} \in \tilde X$ and $d, d_{c} \in \mathrm W_{\mathrm{loc}}^{1,2}([0,\infty);\R^{m})$ the semilinear evolution equation \eqref{eqn:semilinear_eqn} has on some interval $[0,t_{max})$ a unique mild solution \cite[Chapter 6, Theorem 1.4]{pazy2012semigroups}, that is a continuous function that satisfies the integral form of the differential equation \eqref{eqn:semilinear_eqn}. Furthermore, if $z_{0} \in D(\tilde A)$, $d, d_{c} \in C^{2}([0,\infty);\R^{m})$ then the mild solution is a classical solution \cite[Theorem 11.1.3]{Vorabie:03}.

We define the compatibility-regularity set $\D$ for classical solutions as
\begin{align*}
    \D &:= \left\{ (\tilde x_{0},d, d_{c}) \in D(\tilde \A) \times C^{2}([0,\infty);\R^{m})^{2}   
   \text{ with } \tilde \BB \tilde x_{0} = d(0) - S_{c} d_{c}(0) \right\} .
\end{align*}
This leads to the following well-posedness result of the closed-loop system \eqref{syst:closed-loop}.

\begin{prop}\label{prop:mild_solution}
 For every  $(\tilde x_{0},d,d_{c}) \in \D$,
there exists a positive time $t_{max} \le \infty$ such that the closed-loop system \eqref{syst:closed-loop} has a unique classical solution $(\tilde x, (d, d_{c}),y)$ on $[0,t_{max})$ with $\tilde x(0)=\tilde x_0$.
\end{prop}

Next, we show that the closed-loop system \eqref{syst:closed-loop} has global classical and   generalized solutions. 
\begin{prop}\label{prop:prop3.4}
   For every $\tilde x_{0} \in \tilde X$ and $d, d_{c} \in \mathrm L^{2}_{\mathrm{loc}}([0,\infty);\R^{m})$ the closed-loop system \eqref{syst:closed-loop} has a unique generalized solution $(\tilde x_{ext}, (d, d_{c}),y)$ on $[0,\infty)$ with $\tilde x_{ext}(0)=\tilde x_0$. 
    In addition, if $(\tilde x_{0},d,d_{c}) \in \D$, then $\tilde x:=\tilde x_{ext}$ is a classical solution on $[0,\infty)$ and 
\begin{align}
    \dot{\tilde{E}}(\tilde x(s)) 
    &\le \left(|d(s)| + |d_{c}(s)| \|S_c\|\right) |y(s)| -(K v_{2}(s))^\top \RR(K v_{2}(s)) + \|B_{c}\| |K v_{2}(s)| |d_{c}(s)| - \delta |y(s)|^{2}.\label{eqn:derive_energy}
\end{align}
\end{prop}
\begin{proof}
We first assume that $(\tilde x_{0}, d, d_{c}) \in \D$. By Proposition \ref{prop:mild_solution} there exists a positive time $t_{max} \le \infty$ such that the closed-loop system \eqref{syst:closed-loop} has a unique classical solution $(\tilde x, (d, d_{c}),y)$ on $[0,t_{max})$ with $\tilde x(0)=\tilde x_0$. Then the function 
\begin{equation*}
    t \mapsto \tilde E(\tilde x(t)) =  \frac{1}{2} \|x(t)\|^{2}_{X} + \PP(v_{1}(t)) + \frac{1}{2} v_{2}(t)^\top K v_{2}(t)
\end{equation*}
is continuously differentiable and we calculate 
\begin{align}
    \dot{\tilde{E}}(\tilde x(s)) 
    &= \left\langle \A x(s),x(s) \right\rangle_{X} + \dot{v}_{1}(s)^\top \nabla \PP(v_{1}(s)) + v_{2}(s)^\top K \dot{v}_{2}(s)\nonumber \\
    &= \left\langle \A x(s),x(s) \right\rangle_{X} + (K v_{2}(s))^\top \left( -\dot{v}_{2}(s) - \RR(K v_{2}(s)) + B_{c} u_{c}(s) \right) + v_{2}(s)^\top K \dot{v}_{2}(s)\nonumber  \\
    &= \left\langle \A x(s),x(s) \right\rangle_{X} - (K v_{2}(s))^\top \RR(K v_{2}(s)) + (B_{c}^\top K v_{2}(s))^\top u_{c}(s) \nonumber \\
    &= \left\langle \A x(s),x(s) \right\rangle_{X} - (K v_{2}(s))^\top \RR(K v_{2}(s)) + (d(s)- \BB x(s))^\top (y(s) + d_{c}(s)) \nonumber\\
    &\qquad \quad - (S_{c} y(s) + S_{c} d_{c}(s))^\top (y(s) + d_{c}(s)) \nonumber\\
    &= \left\langle \A x(s),x(s) \right\rangle_{X} - (K v_{2}(s))^\top \RR(K v_{2}(s)) + d(s)^\top y(s) - (\BB x(s))^\top y(s) - y(s)^\top S_{c} y(s) - d_{c}(s)^\top S_{c} y(s)\nonumber \\
    &\qquad \quad   + \left(d(s) - (\BB x(s)) - S_{c} y(s) \right)^\top  d_{c}(s) - d_{c}(s)^\top S_{c} d_{c}(s)\nonumber\\
    &=\left\langle \A x(s),x(s) \right\rangle_{X} - (K v_{2}(s))^\top \RR(K v_{2}(s)) + d(s)^\top y(s)  - (\BB x(s))^\top y(s) - y(s)^\top S_{c} y(s) \nonumber \\
    &\qquad \quad - d_{c}(s)^\top S_{c} y(s) + \left(B_{c}^\top K v_{2}(s)\right)^\top d_{c}(s) . \nonumber
\end{align}

Using the fact that system \eqref{syst:PHS} satisfies  \eqref{eqn:impedance_passive}, together with the nonnegativity of $\RR$ and Young's inequality, we obtain that for arbitrary $\lambda>0$ 
\begin{align*}
    \dot{\tilde{E}}(\tilde x(s)) 
    &\le \left(|d(s)| + |d_{c}(s)| \|S_c\|\right) |y(s)| -(K v_{2}(s))^\top \RR(K v_{2}(s)) + \|B_{c}\| |K v_{2}(s)| |d_{c}| - \delta |y(s)|^{2} \\
    &\le \frac{\lambda}{2}|d(s)|^{2} + \left(\frac{\lambda}{2} \|S_c\|^{2} + \frac{1}{2}\right) |d_{c}(s)|^{2} + \left(\frac{1}{\lambda}- \delta \right)|y(s)|^2 + \frac{1}{2}\|B_{c}\|^{2} |K v_{2}(s)|^{2}  .
\end{align*}
Choosing $\lambda := \frac{1}{\delta}$, and applying Grönwall's inequality to
\begin{equation*}
    \dot{\tilde{E}}(\tilde x(s)) 
    \le \|K\|\|B_{c}\|^{2} \tilde E(\tilde x(s)) + \frac{1}{2\delta}|d(s)|^{2} + \left(\frac{1}{2\delta} \|S_c\|^{2} + \frac{1}{2}\right) |d_{c}(s)|^{2} ,
\end{equation*}
yields for $t\in(0,t_{max})$
\begin{equation*}
    \tilde E(\tilde x(t)) 
    \le e^{\|B_{c}\|^{2}\|K\| t} \Bigg(\tilde E(\tilde x_{0}) + \frac{1}{2\delta} \int_{0}^{t} |d(s)|^{2} ds + \left( \frac{1}{2\delta} \|S_{c}\|^{2} + \frac{1}{2} \right) \int_{0}^{t} |d_{c}(s)|^{2} ds \Bigg).
\end{equation*}
Since the energy $\tilde E$ is equivalent to the norm, the solution is uniformly bounded on any finite interval, which by \cite[Chapter 6, Theorem 4.1]{pazy2012semigroups} prevents the solution from blowing up and yields $t_{max}=\infty$.

We note that $\D$ is dense in $\tilde X \times \mathrm L_{\mathrm{loc}}^{2}([0,\infty);\R^{m})^{2}$. Thus, for every $\tilde x_{0} \in \tilde X$ and $d, d_{c} \in \mathrm L_\mathrm{loc}^{2}([0,\infty);\R^{m})$ there exist a sequence $(
\tilde x_{0,n}, d_{n}, d_{c,n}) \in \D$ such that $(\tilde x_{0,n}, d_{n}, d_{c,n})$ converges to $(\tilde x_{0}, d, d_{c})$. 
For each triple $(\tilde x_{0,n},d_{n},d_{c,n}) \in \D$, $n\in \mathbb N$, there exists a  classical solution $\tilde x_{n} \in C([0,\infty);D(\tilde \A)) \cap C^{1}([0,\infty);\tilde X)$ such that for $t >0$
\begin{equation}\label{eqn:eqn3}
\tilde E(\tilde x_{n}(t)) 
    \le e^{\|B_{c}\|^{2} \|K\| t} \left( \tilde E(\tilde x_{0,n}) + \frac{1}{2\delta} \|d_{n|[0,t]}\|^{2} + \left( \frac{1}{2\delta} \|S_{c}\|^{2} + \frac{1}{2} \right) \|d_{c,n|[0,t]}\|^{2} \right).
\end{equation}
Since the energy $\tilde E$ is equivalent to the norm and by taking the limit in inequality \eqref{eqn:eqn3}, we see that the inequality holds for the solution $\tilde x_{ext} = \displaystyle\lim_{n \to \infty} \tilde x_{n}$ \cite[Theorem 3.4]{SchmidZwart:21}. 
Moreover, for $(\tilde x_{0},d,d_{c}) \in \D$, the generalized solution coincides with the classical solution, which completes the proof.
\end{proof}

\section{Input-to-state stability}\label{section4}

In order to achieve input-to-state stability of the closed-loop system we need to impose some extra conditions on the nominal system and on the controller. The assumptions imposed below are adapted from the multiplier framework developed by \cite[Lemma 8.1 and Theorem 8.2]{Aug:19} to establish exponential stability of second-order port-Hamiltonian systems under dissipative boundary feedback. In particular, the coefficient condition is automatically
satisfied when the Hamiltonian density $\HH$ is constant and $P_{0}=P_{1}=0$, while the boundary inequality requires the relevant boundary traces to be controlled by the input and output. Applications to Euler–Bernoulli beams, including conditions adapted to their particular structure, are discussed in \cite[Sections 6.4-6.5]{Aug:16} and \cite[Section 9]{Aug:19}. These assumptions therefore provide established and verifiable sufficient conditions with physically relevant applications, which we use here to derive an ISS estimate in the presence of disturbances.

\begin{assumption}\label{Assumption3}
   For the port-Hamiltonian system \eqref{syst:PHS}, we assume that 
     \begin{enumerate}
         \item\label{item:3.1} there exists $\kappa>0$ such that for $x \in  D(\A)$ 
        \begin{equation*}
            \kappa \left( |(\HH x)(0)|^{2} + |(\HH x)^\prime (0)|^{2}  + | (\HH x)(1)|^{2} \right)  
            \le  |\BB x|^{2} + | \CC x|^{2} .
        \end{equation*}

        \item\label{item:3.2} $\HH^\prime$, $\HH$, $P_{0}$ and $P_{1}$ are chosen such that 
             \begin{align*}
                2 >& \left\| (\zeta - 1)(\HH^\prime(\zeta) \HH^{-1}(\zeta) + P_{2}^{-1} P_{1})\right\|_{ \mathrm L^\infty ((0,1);\R^{n\times n})} \\
                & \qquad + \frac{1}{\surd 2} \left| P_{0}^\top P_{2}^{-1} + P_{2}^{-1} P_{0} P_{1} P_{2}^{-1} P_{1} P_{2}^{-1} \right|  + \frac{1}{2} \left| \left(P_{2}^{-1} P_{0}\right)^\top P_{2}^{-1} P_{0} \right|  .
            \end{align*}
     \end{enumerate}
\end{assumption}

\begin{remark}\label{rem:remark4.2}
\begin{enumerate}
    \item The endpoints $0$ and $1$ play symmetric roles, i.e. Assumption \ref{Assumption3}(\ref{item:3.1}) can be replaced by 
    \begin{equation*}
        \kappa \left( |(\HH x)(1)|^{2} + |(\HH x)^\prime (1)|^{2}  + |(\HH x)(0)|^{2} \right) 
            \le  |\BB x|^{2} + | \CC x|^{2}
    \end{equation*}
    
    \item From Lemma 8.1 of \cite{Aug:19}, it follows that Assumption \ref{Assumption3}(\ref{item:3.2}) implies the existence of a functional $q : X \to \R$ such that there exist $c, \, \hat{c}>0$ such that $|q(x)| \le \hat{c} E(x)$ and
    \begin{align*}
        E(x(t)) + \dot q(x)
        & \le c \left( |(\HH x)(0)|^{2} + | (\HH x)'(0)|^{2} + |(\HH x)(1)|^{2} \right)  .
    \end{align*} 
    
    \item If in addition the port-Hamiltonian system \eqref{syst:PHS} is impedance energy-preserving, then Assumption \ref{Assumption3} implies that it is exactly observable in finite time. For a classical solution $x(\cdot,t)$ of the port-Hamiltonian system with initial condition $x(\cdot,0)=x_0$ and $u=0$, let $q: X \to \R$ be given by Lemma 6.3.3 of \cite{Aug:16}. Then we obtain 
     $|q(x)| \le \hat{c} E(x)$ and
    \begin{align*}
        E(x(t)) + \frac{d}{dt} q(x)
        & \le \frac{c}{\kappa}  |y(t)|^{2}  .
    \end{align*} 
    Thus, for $x_{0} \in D(A)$ 
    and $t>0$
    \begin{equation}\label{eqn:E&y}
        \frac{c}{\kappa} \int_{0}^{t} |y(s)|^{2} ds \ge \int_{0}^{t} E(x(s)) ds - 2\hat{c} E(x_{0})   .
    \end{equation}
    Since the port-Hamiltonian system is impedance energy-preserving, it follows that the energy $t \mapsto E(x(t))$ is constant on $[0,\infty)$. Hence, for $t>0$ 
    \begin{equation*}
        \frac{c}{\kappa} \int_{0}^{t} |y(s)|^{2} ds \ge (t- 2\hat{c}) E(x_{0}) .
    \end{equation*}
    Let $\varepsilon>0$. Choosing $t_{0} := 2\hat{c} +\varepsilon$ and $c_{obs}:=\varepsilon\frac{\kappa}{c}$, for $t \in [t_{0}, \infty)$ we obtain the  observability inequality 
    \begin{equation*}
        \int_{0}^{t} |y(s)|^{2} ds \ge c_{obs} \|x_0\|^{2}_X,
    \end{equation*}
   which implies that the port-Hamiltonian system \eqref{syst:PHS} is exactly observable in finite time.

   \item If $P_{1}=P_{0}=0$, then Assumption \ref{Assumption3}(\ref{item:3.2}) is replaced by, for some $\varepsilon>0$ 
   $$ \HH(\zeta)- (1-\zeta) \HH'(\zeta) \ge \varepsilon I , $$  which is always true if $\HH$ is constant \cite[Lemma 4.3.17]{Aug:16}.

   \item For the Euler-Bernoulli beam model, which can be formulated as a port-Hamiltonian system of the form
   \begin{equation*}
       \HH = \begin{bmatrix}
           \HH_{1} & 0 \\
           0 & \HH_{2}
       \end{bmatrix} , \qquad  P_{2} = \begin{bmatrix}
           0 & -P^\top \\
           P & 0
       \end{bmatrix} , \qquad P_{1} = 0,
   \end{equation*}
   where $P \in \R^{\frac{n}{2} \times \frac{n}{2}}$ is an invertible matrix, Assumption \ref{Assumption3}(\ref{item:3.1}) can be replaced by \cite[Lemma 6.4.1]{Aug:16}
   \begin{equation*}
       \kappa |(\HH x)(0)|^{2} + |(\HH x_{1})'(0)|^{2} + |(\HH x_{2})(1)|^{2} \le  |\BB x|^{2} + | \CC x|^{2} .
   \end{equation*}
\end{enumerate}
\end{remark}

\medskip
The following theorem provides sufficient conditions to obtain input-to-state stability of the closed-loop system \eqref{syst:closed-loop}. To this end, we impose the following assumptions on the finite-dimensional nonlinear system \eqref{syst:controller}.

\begin{assumption}\label{Assumption4}
    \begin{enumerate}  
        \item There exist constants $\overline{c}_{1},\underline{c}_{1}>0$ such that for $v_{1} \in \R^{n_{c}}$
        \begin{equation*}
            \overline{c}_{1} v_{1}^\top \nabla \PP(v_{1}) \ge \PP(v_{1}) \ge \underline{c}_{1} |v_{1}|^{2} .
        \end{equation*}
        
        \item There exist $\overline{c}_{2}, \underline{c}_{2}>0$ such that for $v_{2}\in \R^{n_{c}}$ 
        \begin{equation*}
            \overline{c}_{2} v_{2}^\top \RR(v_{2}) \ge |v_{2}|^{2} \ge \underline{c}_{2} |\RR(v_{2})|^{2} .
        \end{equation*}
    \end{enumerate}
\end{assumption}

\begin{thm}\label{thm:main}
    Suppose that Assumptions \ref{Assumption3} and \ref{Assumption4} are satisfied. Then the closed-loop system \eqref{syst:closed-loop} is input-to-state stable, i.e. there exist a time $T_{0}>0$, $\beta_{0} \in (0,1)$ and $C_1, C_{2}>0$
    such that for all $\tilde x_{0} \in \tilde X$ and $d, d_{c} \in \mathrm L_{\mathrm{loc}}^{2}([0,\infty);\R^{m})$, the generalized solution to \eqref{syst:closed-loop} satisfies
    \begin{equation*}
        \| \tilde x_{ext}(t)\|_{\tilde X} \le \beta\left(\|\tilde x_{0}\|_{\tilde X}, t \right) + \gamma_{1} \left(\|d_{|[0,t]}\|_{\mathrm L^{2}} \right) + \gamma_{2} \left(\|d_{c|[0,t]} \|_{\mathrm L^{2}} \right) ,\quad t>0 .
    \end{equation*}
    where
    \begin{equation*}
    \beta: (r,t) \mapsto \phi_{1}^{-1} \left(\beta_{0}^{\frac{t}{T_{0}}} \phi_{2}(r) \right) \in \mathcal{KL}  \qquad 
    \gamma_{1} : r \mapsto \phi_{1}^{-1} \left( 3C_{1} r^{2} \right) \in \mathcal K_{\infty} \qquad
    \gamma_{2} : r \mapsto  \phi_{1}^{-1} \left( 3C_{2} r^{2}\right) \in \mathcal K_{\infty} ,
   \end{equation*}
   where $\phi_{1}$ and $\phi_{2}$ are the comparison functions introduced in \eqref{eqn:energy_equivalence}.
\end{thm}
To prove Theorem \ref{thm:main}, we first establish the following lemma.
\begin{lemma}\label{lem:lemma2}
Suppose that  Assumptions \ref{Assumption3} and \ref{Assumption4} hold. Then there exists a strictly decreasing function $\rho:(0,\infty)\to \R^{+}$ with $\displaystyle\lim_{t\to\infty} \rho(t)=0$, such that for every $(\tilde x_{0},d,d_{c}) \in \D$ the corresponding classical solution $(\tilde x, (d,d_c), y)= ((x,v_1, v_2), (d,d_c), y)$ satisfies for all $t>0$
\begin{align*}
    \tilde E(\tilde x(t))
    &\le \rho(t) \left( \tilde E(\tilde x_{0}) + \int_{0}^{t} |d(s)|^{2} ds + \int_{0}^{t} |y(s)|^{2} ds \right) \\ 
    &\qquad \quad + \left(1+\rho(t)\right) \left( \|B_{c}\|^{2} \int_{0}^{t} |Kv_{2}(s)|^{2} ds +  \int_{0}^{t} |d_{c}(s)|^{2} ds + \int_{0}^{t} \left(|d(s)| + \|S_{c}\| |d_{c}(s)| \right) |y(s)| ds \right) .
\end{align*}
\end{lemma}
\begin{proof}
Equation \eqref{eqn:derive_energy} shows 
\begin{align*}
    \dot{\tilde{E}}(\tilde x(s)) 
    &\le \left(|d(s)| + \|S_{c}\| |d_{c}(s)| \right) |y(s)|  + \|B_{c}\| |K v_{2}(s)| |d_{c}| - \delta |y(s)|^{2} \\
    &\le\left(|d(s)| + \|S_{c}\| |d_{c}(s)| \right) |y(s)|  
    + \frac{1}{2}\|B_{c}\|^{2} |K v_{2}(s)|^{2} + \frac{1}{2}|d_{c}(s)|^{2} .
\end{align*}
Let $t \ge r\ge 0$. Integration over the interval $[r,t]$ implies 
\begin{equation}\label{eqn:eqn1}
    \tilde E(\tilde x(t)) 
    \le \tilde E(\tilde x(r)) + \int_r^t \left(|d(s)| + \|S_{c}\| |d_{c}(s)| \right) |y(s)| ds + \frac{1}{2} \|B_{c}\|^{2} \int_{r}^{t} |Kv_{2}(s)|^{2} ds   + \frac{1}{2}  \int_{r}^{t} |d_{c}(s)|^{2}ds . 
\end{equation}
Hence,  for $t\ge 0$ we get  
\begin{align}
     t \tilde E(\tilde x(t)) 
     &= \int^{t}_{0} \tilde E(\tilde x(t)) dr \nonumber\\
     &\le \int^{t}_{0} \tilde E(\tilde x(r)) dr  + \int^{t}_{0} \int_{r}^{t} \left(|d(s)| + \|S_{c}\| |d_{c}(s)| \right) |y(s)| ds dr + \frac{1}{2} \|B_{c}\|^{2} \int^{t}_{0} \int_{r}^{t} |Kv_{2}(s)|^{2} ds dr \nonumber\\
     &\qquad \quad  + \frac{1}{2} \int^{t}_{0} \int_{r}^{t} |d_{c}(s)|^{2} ds dr  \nonumber \\
     &\le \int^{t}_{0} \tilde E(\tilde x(s)) ds  + t \int^{t}_{0} \left(|d(s)| + \|S_{c}\| |d_{c}(s)| \right) |y(s)| ds  + \frac{t}{2} \|B_{c}\|^{2} \int^{t}_{0}  |Kv_{2}(s)|^{2} ds \label{eqn:t_E(t)} \\
     &\qquad \quad + \frac{t}{2} \int^{t}_{0} |d_{c}(s)|^{2} ds  \nonumber .
\end{align}
On the other hand, for every $(\tilde x_{0},d, d_{c}) \in \D$, Remark \ref{rem:remark4.2} provides a functional $q: X \to \R$ such that for $x \in X$, $|q(x)| \le \hat{c} \|x\|^{2}_{X}$ and
\begin{equation}\label{eqn:q(x(t))}
    E(x(t)) + \dot q(x)
    \le c \left( |(\HH x)(0,t)|^{2} + |(\HH x)^\prime (0,t)|^{2} + |(\HH x)(1,t)|^{2} \right) .
\end{equation}
In addition, under Assumption \ref{Assumption4} and from \cite[Lemma 17]{RamirezZwartLeGorrec:17}, it is easy to verify that for some constant $c_{0}>0$, the energy of the controller satisfies
\begin{equation*}
    \int_{0}^{t} E_{c}(v(s)) ds 
    \le c_{0} \left( E_{c}(v_{0}) + \int_{0}^{t} |y(s)|^{2} ds + \int_{0}^{t} |d_{c}(s)|^{2} ds \right) .
\end{equation*}
Using equation \eqref{eqn:q(x(t))}, the energy of the closed-loop system satisfies
\begin{align*}
     \int_{0}^{t} \tilde E(\tilde x(s)) ds
     &= \int_{0}^{t} E(x(s)) ds + \int_{0}^{t} E_{c}(v(s)) ds \\
     &\le q(0) - q(x) + c  \int_{0}^{t} \left( |(\HH x)(0,s)|^{2}  + |(\HH x)^\prime (0,s)|^{2} ds +  |(\HH x)(1,s)|^{2} \right) ds   \\ 
     &\qquad +  c_{0} \left( E_{c}(v_{0}) + \int_{0}^{t} |y(s)|^{2} ds + \int_{0}^{t} |d_{c}(s)|^{2} ds \right) \\
     &\le \hat{c} \tilde E(\tilde x_{0}) + \hat{c} \tilde E(\tilde x(t)) + \frac{c}{\kappa} \left( \int_{0}^{t} |u(s)|^{2} ds + \int_{0}^{t} |y(s)|^{2} ds \right) \\
     &\qquad \quad +  c_{0} \left( E_{c}(v_{0}) + \int_{0}^{t} |y(s)|^{2} ds + \int_{0}^{t} |d_{c}(s)|^{2} ds \right),
\end{align*}
where we used  Assumption \ref{Assumption3} and  $|q(x)|\le \hat{c} E(x(t))$ in the last estimate.
Using again \eqref{eqn:eqn1} and the fact that $E_{c}(v_{0})\le \tilde E(\tilde x_{0})$, it follows
\begin{equation}\label{eqn:integral_E(s)}
\begin{aligned}
    \int_{0}^{t} \tilde E(\tilde x(s)) ds
    &\le (2\hat{c} + c_{0}) \tilde E(\tilde x_{0}) + \hat{c} \int_{0}^{t} (|d(s)| + \|S_{c}\| |d_{c}(s)|) |y(s)| ds + \frac{\hat{c}}{2} \|B_{c}\|^{2} \int_{0}^{t} |Kv_{2}(s)|^{2} ds  \\
     &\qquad \quad  + (\frac{\hat c}{2} + c_0) \int_{0}^{t} |d_{c}(s)|^{2} ds +  \frac{c}{\kappa} \int_{0}^{t} |u(s)|^{2} ds + (\frac{c}{\kappa}+c_0)\int_{0}^{t} |y(s)|^{2} ds. 
\end{aligned}
\end{equation}
Substituting \eqref{eqn:integral_E(s)} into \eqref{eqn:t_E(t)} and together with 
\begin{align*}
    \int_{0}^{t} |u(s)|^{2} ds 
    &= \int_{0}^{t} |-B_{c}^\top K v_{2}(s) - S_{c}  y(s) -S_{c} d_{c}(s) + d(s)|^{2} ds \\
    &\le 4\|B_{c}\|^{2} \int_{0}^{t} |K v_{2}(s)|^{2} ds + 4\|S_{c}\|^{2} \int_{0}^{t} \left( |y(s)|^{2} + |d_{c}(s)|^{2} \right)ds + 4\int_{0}^{t} |d(s)|^{2} ds,
\end{align*}
we obtain
\begin{align*}
    t \tilde E(\tilde x(t))
     &\le  (2\hat{c} +c_{0}) \tilde E(\tilde x_{0}) + (t +\hat{c}) \int_{0}^{t} \left(|d(s)| + \|S_{c}\||d_{c}(s)| \right) |y(s)| ds + \frac{4c}{\kappa} \int_{0}^{t} |d(s)|^{2} ds  \\
     &\qquad  \quad + \left(\frac{t}{2} + \frac{\hat c}{2} + \frac{4c}{\kappa} \right) \|B_{c}\|^{2} \int_{0}^{t} |Kv_{2}(s)|^{2} ds +\left( \frac{t}{2} + \frac{\hat c}{2} + c_{0} + \frac{4c}{\kappa} \|S_{c}\|^{2} \right) \int_{0}^{t} |d_{c}(s)|^{2} ds  \\
     &\qquad \quad + \left(\frac{4c}{\kappa} \|S_{c}\|^{2} + \frac{c}{\kappa} + c_{0} \right) \int_{0}^{t} |y(s)|^{2} ds .
\end{align*}
Therefore, for $t >0$, taking $\rho(t):= \frac{1}{t} \max\left\{ 2 \hat{c} +c_{0}, \, \frac{4c}{\kappa} +\frac{\hat c}{2}, \, \frac{4c}{\kappa}\|S_{c}\|^{2} +c_{0} +\frac{\hat c}{2} +\frac{c}{\kappa} \right\}$ we have $\rho$ strictly decreasing with $\displaystyle\lim_{t\to\infty} \rho(t)=0$ and
\begin{align*}
    \tilde E(\tilde x(t))
     &\le \rho(t) \left(\tilde E(\tilde x_{0}) +  \int_{0}^{t} |d(s)|^{2} ds + \int_{0}^{t} |y(s)|^{2} ds \right) +(\frac{1}{2}+\rho(t)) \left(\|B_{c}\|^{2} \int_{0}^{t} |Kv_{2}(s)|^{2} ds +  \int_{0}^{t} |d_{c}(s)|^{2} ds\right) \\
     &\qquad \quad +(1+\rho(t)) \int_{0}^{t} \left(|d(s)| + \|S_{c}\| |d_{c}(s)| \right) |y(s)| ds .
\end{align*}
\end{proof}

Now, we are prepared to give the proof of the main result. 
\medskip

\noindent\textit{Proof of Theorem \ref{thm:main}}.
Using equation \eqref{eqn:derive_energy} with Assumption \ref{Assumption4}, we have for $\lambda>0$, $(\tilde x_{0},d,d_{c}) \in \D$ and $s\ge 0$
\begin{align}
    \dot{\tilde{E}}(\tilde x(s)) 
    &\le \left(|d(s)| + \|S_{c}\| |d_{c}(s)| \right) |y(s)| -(K v_{2}(s))^\top \RR(K v_{2}(s)) + \|B_{c}\| |K v_{2}(s)| |d_{c}| - \delta |y(s)|^{2} \nonumber \\ 
    &\le\left(|d(s)| + \|S_{c}\||d_{c}(s)| \right) |y(s)| -\frac{1}{\overline{c}_{2}}|K v_{2}(s)|^{2} 
    + \frac{1}{2\lambda}\|B_{c}\|^{2} |K v_{2}(s)|^{2} + \frac{\lambda}{2}|d_{c}(s)|^{2} - \delta |y(s)|^{2} . \label{eqn:eq4.6} 
\end{align}
Thus, by integrating over $[0,t]$ and for $\varepsilon>0$
\begin{equation}\label{eqn:eqn2}
\begin{aligned}
    \tilde E(\tilde x(t)) 
    &\le \tilde E(\tilde x_{0}) + \int_{0}^{t} \left(|d(s)| + \|S_{c}\| |d_{c}(s)| \right) |y(s)| ds + \left(\frac{1}{2\lambda} \|B_{c}\|^{2} -\frac{1}{\overline{c}_{2}} \right) \int_{0}^{t} |K v_{2}(s)|^{2} ds \\
     & \qquad \quad + \frac{\lambda}{2} \int_{0}^{t} |d_{c}(s)|^{2}ds  - \delta(1-\varepsilon) \int_{0}^{t} |y(s)|^{2} ds -\delta\varepsilon \int_{0}^{t} |y(s)|^{2} ds .
\end{aligned}
\end{equation}
By Lemma \ref{lem:lemma2}, we have that 
\begin{align*}
    - \int_{0}^{t} |y(s)|^{2} ds
     &\le -\frac{1}{\rho(t)} \tilde E(\tilde x(t)) + \tilde E(\tilde x_{0}) + \int_{0}^{t} |d(s)|^{2} ds + \frac{1+\rho(t)}{\rho(t)} \|B_{c}\|^{2} \int_{0}^{t} |Kv_{2}(s)|^{2} ds \\
     &\qquad \quad + \frac{1+\rho(t)}{\rho(t)} \int_{0}^{t} |d_{c}(s)|^{2} ds + \frac{1+\rho(t)}{\rho(t)} \int_{0}^{t} \left(|d(s)| + \|S_c\||d_{c}(s)| \right) |y(s)| ds .
\end{align*}
Plugging this into equation \eqref{eqn:eqn2}, we obtain for $\varepsilon>0$ 
\begin{align*}
    \tilde E(\tilde x(t)) 
    &\le \tilde E(\tilde x_{0}) +\left(1+ \frac{\delta \varepsilon(1+\rho(t))}{\rho(t)}\right) \int_{0}^{t} \left(|d(s)| + \|S_{c}\| |d_{c}(s)| \right) |y(s)| ds - \delta(1-\varepsilon) \int_{0}^{t} |y(s)|^{2} ds   \\
    &\qquad \quad  - \frac{\delta\varepsilon}{\rho(t)} \tilde E(\tilde x(t)) + \delta\varepsilon \tilde E(\tilde x_{0}) + \delta\varepsilon \int_{0}^{t} |d(s)|^{2} ds + \left(\frac{\lambda}{2}+ \frac{\delta \varepsilon(1+\rho(t))}{\rho(t)} \right) \int_{0}^{t} |d_{c}(s)|^{2} ds\\
    &\qquad \quad + \left( \frac{\delta \varepsilon(1+\rho(t))}{\rho(t)} \|B_{c}\|^{2} + \frac{1}{2\lambda}\|B_{c}\|^{2} -\frac{1}{\overline{c}_{2}} \right) \int_{0}^{t} |K v_{2}(s)|^{2} ds.
\end{align*}
Applying Young's inequality and regrouping terms, we obtain
\begin{align*}
    \left( 1+ \frac{\delta \varepsilon}{\rho(t)} \right)\tilde E(\tilde x(t)) 
    &\le \left(1 + \delta \varepsilon \right) \tilde E(\tilde x_{0}) + \left( \frac{\alpha}{2}\left(1+\frac{\delta \varepsilon(1+\rho(t))}{\rho(t)} \right) + \delta\varepsilon \right) \int_{0}^{t} |d(s)|^{2} ds \\
     &\qquad + \left(\frac{\lambda}{2}+ \frac{\alpha}{2} \left(1+ \frac{\delta \varepsilon(1+\rho(t))}{\rho(t)} \right)\|S_{c}\|^{2} + \frac{\delta \varepsilon(1+\rho(t))}{\rho(t)} \right) \int_{0}^{t} |d_{c}(s)|^{2} ds \\
    &\qquad + \left( \frac{\delta \varepsilon(1+\rho(t))}{\rho(t)}\|B_{c}\|^{2} + \frac{1}{2\lambda}\|B_{c}\|^{2} -\frac{1}{\overline{c}_{2}} \right) \int_{0}^{t} |K v_{2}(s)|^{2} ds \\
    &\qquad +\left(\frac{1}{\alpha}\left(1+ \frac{\delta \varepsilon(1+\rho(t))}{\rho(t)} \right) + \delta\varepsilon -\delta  \right) \int_{0}^{t} |y(s)|^{2} ds ,
\end{align*}
for all $t>0$ and arbitrary $\lambda,\alpha>0$. In particular, this holds for 
$$ \frac{1}{\delta} < \alpha < 2 C_{1}, \qquad
    \frac{\overline{c}_{2}}{2} \|B_{c}\|^{2} < \lambda < 2 C_{2} -\frac{1}{\delta}\|S_{c}\|^{2}, $$
where $C_{1}, C_{2}>0$ are constants such that 
\begin{equation*} 
    C_{2} > \frac{\overline{c}_{2}}{4} \|B_{c}\|^{2} +\frac{1}{2\delta}\|S_{c}\|^{2}, 
    \qquad C_{1} > \frac{1}{2\delta}. 
\end{equation*}
Now, let $T_{0} \in (0,\infty)$ be such that $\rho(T_{0})<1$ (this is possible because $\displaystyle\lim_{t\to\infty}\rho(t)=0$). We define 
\begin{equation*}
    \beta_{0} := \frac{1+ \delta \varepsilon}{1+ \frac{\delta \varepsilon}{\rho(T_{0})}} <1 ,
\end{equation*}
and we choose $\varepsilon \in (0,1)$ so small such that
\begin{align*}
    \frac{\delta \varepsilon(1+\rho(T_0))}{\rho(T_0)}\|B_{c}\|^{2} + \frac{1}{2\lambda}\|B_{c}\|^{2} &\le \frac{1}{\overline{c}_{2}}, \qquad \quad
    \frac{1}{\alpha} + \delta\varepsilon\left( 1 + \frac{1+ \rho(T_{0})}{\alpha \rho(T_{0})}\right)
    \le \delta,  \\
    \frac{\alpha}{2} +\left(1 + \frac{\alpha}{2}\right) \frac{\delta\varepsilon}{1 +\frac{\delta \varepsilon}{\rho(T_{0})}} &\le C_{1}, \qquad \quad
    \frac{\alpha}{2}\|S_{c}\|^{2} + \frac{\lambda}{2 +\frac{2\delta \varepsilon}{\rho(T_{0})}} + \left(1 + \frac{\alpha}{2} \|S_{c}\|^{2} + \frac{1}{\rho(T_{0})}\right) \frac{\delta\varepsilon}{1 +\frac{\delta \varepsilon}{\rho(T_{0})}} 
    \le C_{2} .
\end{align*}
Hence, we have
\begin{align*}
    \tilde E(\tilde x(T_{0})) 
    &\le \beta_{0} \tilde E(\tilde x_{0}) + C_{1} \int_{0}^{T_{0}} |d(s)|^{2} ds + C_{2} \int_{0}^{T_{0}} |d_{c}(s)|^{2} ds .
\end{align*}
Moreover, since $\frac{\overline{c_2}}{2} \|B_{c}\|^{2} < \lambda <2C_{2} -\frac{1}{\delta}\|S_{c}\|^{2} $, it follows from \eqref{eqn:eq4.6} that for $s\ge0$
\begin{equation*}
     \dot{\tilde{E}}(\tilde x(s)) 
    \le\left(|d(s)| + \|S_{c}\||d_{c}(s)| \right) |y(s)| + C_{2} |d_{c}(s)|^{2} .
\end{equation*}
For every $n \in \N_{0}$ and $(\tilde x_{0},d,d_{c}) \in \D$, by induction we obtain
\begin{align*}
    \tilde E(\tilde x(nT_{0})) 
    &\le \beta_{0}^{n} \tilde E(\tilde x_{0}) + C_{1} \|d_{|[0,nT_{0}]}\|^{2}_{\mathrm L^2}  + C_{2} \|d_{c|[0,nT_{0}]}\|^{2}_{\mathrm L^2} . 
\end{align*}
Let $n = n(t) \in \N_{0}$ such that $t- nT_{0} \in [0,T_{0})$. Taking $t= nT_{0} + (t - nT_{0})$, and by the cocycle property we get 
\begin{align*}
    \tilde E(\tilde x(t)) 
    &\le \beta_{0}^{n} \tilde E(\tilde x(t-nT_{0})) + C_{1} \|d_{|[t-nT_{0},t]}\|^{2}_{\mathrm L^2}  + C_{2} \|d_{c|[t-nT_{0},t]}\|^{2}_{\mathrm L^2} \\
    &\le \beta_{0}^{n} \tilde E(\tilde x_{0}) + C_{1} \|d_{[0,t]}\|^{2}_{\mathrm L^{2}}  + C_{2} \|d_{c|[0,t]}\|^{2}_{\mathrm L^2} .
\end{align*}
Thus, for every $(\tilde x_{0}, d, d_{c}) \in \D$ 
\begin{equation*}
    \tilde E(\tilde x(t))  
    \le \beta_{0}^{\frac{t}{T_{0}}} \tilde E(\tilde x_{0}) + C_{1} \|d_{|[0,t]}\|^{2}_{\mathrm L^2}  + C_{2} \|d_{c|[0,t]}\|^{2}_{\mathrm L^2}, \quad t>0 .
\end{equation*}
By equation \eqref{eqn:energy_equivalence}, $\tilde E(\tilde x)$ is equivalent to $\|\tilde x\|$ in the sense that there exist $\phi_{1}$, $\phi_{2} \in \K_{\infty}$ such that for all $\tilde x \in \tilde X$
\begin{equation*}
    \phi_{1}(\|\tilde x\|) \le \tilde E(\tilde x) \le \phi_{2}(\| \tilde x\|). 
\end{equation*}
We conclude that for every $(\tilde x_{0},d,d_{c}) \in \D$ and $t>0$
\begin{equation*}
    \|\tilde x(t)\|_{\tilde X} \le \beta(\|\tilde x_{0}\|_{\tilde X},t) + \gamma_{1} \left(\|d_{|[0,t]}\|_{\mathrm L^2} \right) + \gamma_{2} \left(\|d_{c|[0,t]}\|_{\mathrm L^2} \right),
\end{equation*}
where
\begin{align*}
    \beta &: (r,t) \mapsto \phi_{1}^{-1} \left(3\beta_{0}^{\frac{t}{T_{0}}} \phi_{2}(r) \right) \in \mathcal{KL} \qquad 
    \gamma_{1} : r \mapsto \phi_{1}^{-1} \left( 3 C_{1} r^{2} \right) \in \K_{\infty} \qquad
    \gamma_{2} : r \mapsto  \phi_{1}^{-1} \left( 3 C_{2} r^{2}\right) \in \K_{\infty} \\
    C_{2} &> \frac{\overline{c}_{2}}{4} \|B_{c}\|^{2} + \frac{1}{2\delta}\|S_{c}\|^{2}, \qquad C_{1} > \frac{1}{2\delta} \cdot
\end{align*}
Moreover, using the same argument as in the proof of Proposition \ref{prop:prop3.4}, we conclude that the inequality similarly holds for $\tilde x_{0} \in \tilde X$ and $d, d_{c} \in \mathrm L_{\mathrm{loc}}^{2}([0,\infty);\R^{m})$.
\qed

\begin{remark}
    In the input-to-state inequality, the decay term is exponential when expressed in terms of the energy $\tilde E$. After rewriting the estimate in terms of the state norm, the corresponding function $\beta$ can also be chosen exponential in time, but its dependence on the initial norm is not necessarily linear. Thus, this estimate does not necessarily imply global exponential stability in the state norm.
\end{remark}

\section{Application: Euler-Bernoulli beam}\label{section5}
We consider an Euler–Bernoulli beam clamped at the left
end and actuated at the right end, whose vibrations are described by 
\begin{equation}\label{eqn:EulerBeam}
    \rho \frac{\partial^2 \omega}{\partial t^2}(\zeta,t) + EI \frac{\partial^4 \omega}{\partial \zeta^4}(\zeta,t) =0, \quad t>0, \quad \zeta \in (0,1), 
\end{equation}
with the initial conditions
\begin{equation*}
    \omega(\zeta,0) = \omega^{0}(\zeta), \qquad \frac{\partial\omega}{\partial t}(\zeta,0) = \omega^{1}(\zeta) .
\end{equation*}
Here, $\omega(\zeta,t)$ denotes the vertical displacement, the positive constant $\rho$ is the mass density per unit length, and $EI > 0$ the bending stiffness of the beam, where $E> 0$ is the modulus of elasticity and $I > 0$ is the second moment of area of the cross section.
We define  
$$ x=\begin{bsmallmatrix}
    \rho \frac{\partial \omega}{\partial t} \\ \frac{\partial^2 \omega}{\partial \zeta^2}
\end{bsmallmatrix}, \qquad P_2= \begin{bmatrix}
    0 & -1 \\ 
    1 & 0
\end{bmatrix}, \qquad P_1 =0, \qquad P_0= 0, \text{ and} \qquad \mathcal{H}=\begin{bmatrix}
    \frac{1}{\rho} & 0 \\
    0 & E I
\end{bmatrix}. $$
Then the Euler–Bernoulli beam equations \eqref{eqn:EulerBeam} can be written in the form of the system class \eqref{syst:PHS} with boundary conditions given by
\begin{align*}
    \HH_{1} x_{1}(0,t) &= \HH_{1} x_{1}^\prime(0,t) = \HH_{2} x_{2}(1,t) = 0 , \\ 
    -\HH_{2} x_{2}^\prime(1,t) &= u(t), \qquad y(t) = \HH_{1} x_{1}(1,t) .
\end{align*}
To stabilize the beam vibrations, we connect the tip of the beam to a nonlinear dynamic controller described by the finite-dimensional system \eqref{syst:controller} with 
$$ \PP(v_{1})= \frac{1}{2} \|v_{1}\|^{2}+ \frac{1}{4} \|v_{1}\|^{4}, \quad \RR(v_{2})= v_{2} , \quad K=B_{c}=1, \quad \text{and} \quad S_{c}=\delta . $$
The potential represents a hardening spring, while the damping is linear and viscous.  
The state variable $v =\begin{psmallmatrix}
    v_1 \\ v_2
\end{psmallmatrix} \in \mathbb{R}^2$, where $v_1$ and $v_2$ represent virtual internal displacement and momentum, respectively. This finite-dimensional system  can be interpreted as a passive vibration-control device coupled to the beam tip.

Note that $\PP$ and $\RR$ satisfy Assumption \ref{Assumption4} and the Euler-Bernoulli beam satisfies Assumption \ref{Assumption3}. Thus, by Theorem \ref{thm:main}, we obtain that the closed-loop system is input-to-state stable.

\begin{lemma}\cite[Lemma 2.1 ]{CrepeauPrieur2006}
    The eigenfunctions $(\psi_{n})_{n\ge 1}$  of the operator 
    \begin{align*}
        \A x &= \frac{EI}{\rho}  \frac{\partial^4 x}{\partial \zeta^4}, \\
        D(\A) &= \{ \omega \in \mathrm H^{4}(0,1)\mid  \omega(0) = \frac{\partial \omega}{\partial \zeta}(0) = \frac{\partial^2 \omega}{\partial \zeta^2}(1) = \frac{\partial^3 \omega}{\partial \zeta^3}(1) = 0 \},
    \end{align*}
    are given by
    \begin{align*}
        \psi_{n}(\zeta) &= \cos(\alpha_{n}\zeta) - \cosh(\alpha_{n}\zeta) + \mu_{n} \left( \sinh(\alpha_{n}\zeta) - \sin(\alpha_{n} \zeta) \right) , \qquad \mu_{n} = \frac{\cosh(\alpha_{n})+ \cos(\alpha_{n})}{\sinh(\alpha_{n})+ \sin(\alpha_{n})}, 
    \end{align*}
    where $\alpha_{n}$ is the $n$-th positive root of
    $ 1 + \cos(\alpha_{n}) \cosh(\alpha_{n}) = 0$, and the corresponding eigenvalues $(\lambda_{n})_{n\ge 1}$
    are given by 
    $$  \lambda_{n} = \frac{EI}{\rho} \alpha_{n}^{4} . $$
\end{lemma}
\noindent
The functions $(\psi_{n})_{n\ge1}$ form an orthonormal basis of $\mathrm L^{2}(0,1)$.
By expanding $\omega$ in terms of $(\psi_{n})_{n\ge1}$, that is, $\omega(\zeta,t) = \displaystyle\sum_{n=1}^{\infty} \omega_{n}(t) \psi_{n}(\zeta)$,
we obtain the following ODE for each $n$
\begin{equation*}
    \ddot{\omega}_{n} + \lambda_{n}\omega_{n} = \frac{1}{\rho} \psi_{n}(1) u(t) . 
\end{equation*}
Using the initial conditions, it follows
\begin{equation*}
    \omega_{n}(t) = \omega^{0}_{n} \cos \left(\sqrt{\frac{EI}{\rho}}\alpha_{n}^{2}t\right) + \frac{\omega^{1}_{n}}{\sqrt{\frac{EI}{\rho}}\alpha_{n}^{2}} \sin\left(\sqrt{\frac{EI}{\rho}}\alpha_{n}^{2}t\right) + \frac{1}{\sqrt{\frac{EI}{\rho}}\alpha_{n}^2\rho} \psi_{n}(1) \int_{0}^{t} \sin\left(\sqrt{\frac{EI}{\rho}}\alpha_{n}^{2}(t-s)\right) u(s) ds   .
\end{equation*}
For the numerical result, we have taken $\rho=EI=1$, $\delta=1$, $N=20$ and as initial conditions 
\begin{equation*}
    \omega^{0}(\zeta) = 0.8(0.05\psi_{1}(\zeta) + 0.03 \psi_{2}(\zeta)), \qquad
    \omega^{1}(\zeta) = 0.02\psi_{1}(\zeta), \text{ and} \qquad v_{0}=0 .
\end{equation*}
The disturbances at the plant and controller inputs are given by
\begin{equation*}
    d(t) = 0.7 \sin(5t), \qquad d_{c}(t) = 0.5 \cos(2t) ,
\end{equation*}
and compute the response shown in Figure \ref{fig:beam_iss}. 

\begin{figure}
    \centering
    \includegraphics[width=0.5\linewidth]{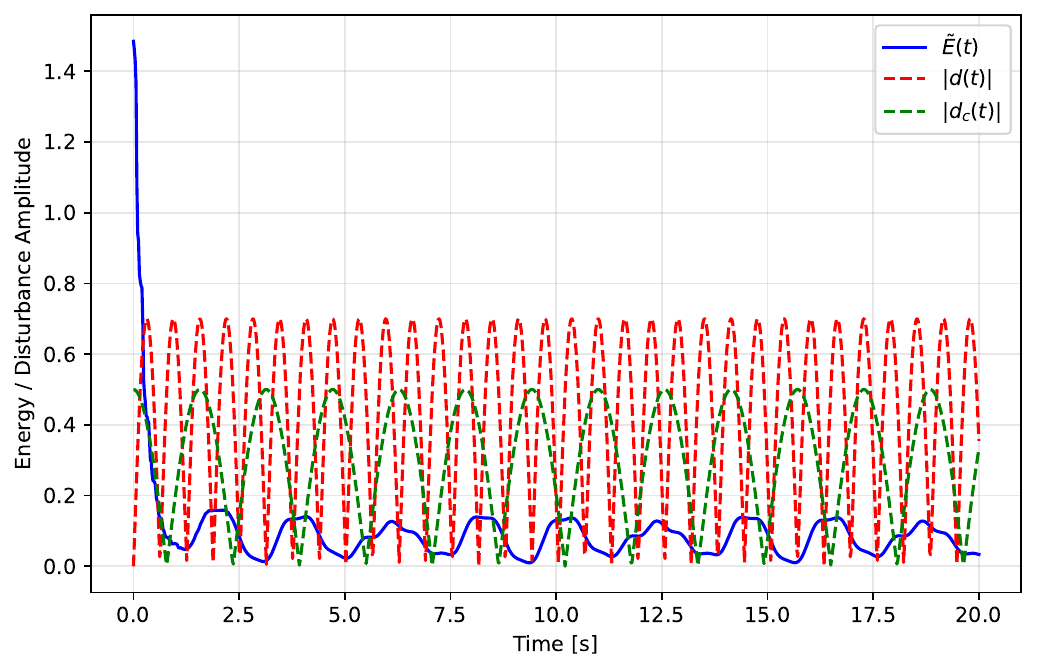}
    \caption{Energy-based ISS simulation.}
    \label{fig:beam_iss}
\end{figure}

\section{Conclusion}\label{section6}
In this paper, we studied the stabilization of a class of infinite-dimensional port-Hamiltonian systems under nonlinear boundary control in the presence of external actuator disturbances at the plant and controller inputs. We established the well-posedness of the resulting infinite-dimensional dynamics using semigroup theory and the theory of semilinear evolution equations. Under the stated sufficient conditions on the port-Hamiltonian system and the controller, we derived an input-to-state stability property for the closed-loop system by combining the energy balance of the port-Hamiltonian formulation with the dissipative properties of the finite-dimensional component.

The theoretical results are illustrated through the Euler--Bernoulli beam. This example demonstrates the applicability of the proposed framework to flexible structures, where the finite-dimensional passive subsystem can represent either an intrinsic component of the physical system, such as a tip mass, a mass-spring-damper system, or an attached oscillator, or a passive control element, such as a vibration absorber or an energy-dissipating controller.

Moreover, the obtained input-to-state stability result is not restricted to the particular class of examples considered in this paper. The stability argument also applies to higher-order port-Hamiltonian systems
\begin{equation*}
    \frac{\partial x}{\partial t}(\zeta,t) = \sum_{k=0}^{N} P_k \frac{\partial^k}{\partial \zeta^k} \HH(\zeta) x(\zeta,t) , \quad t>0, \quad \zeta \in (0,1),
\end{equation*}
provided that there exists a functional $q$ satisfying $|q(x)| \le \hat{c} \|x\|^{2}$ and 
\begin{equation*}
    E(x(t)) + \frac{d}{dt} q(x) \le c \left( |u(t)|^{2} + |y(t)|^{2} \right),
\end{equation*}
along classical solution, for some constant $\hat{c}, c>0$. Under the same assumptions on the controller, the proof then yields the corresponding ISS estimate. 
As a perspective, an interesting extension of this work is to consider more general classes of nonlinear finite-dimensional controllers beyond the strictly passive framework investigated here.

\section*{Acknowledgments}
This work was supported and funded by the European Union (Horizon Europe MSCA project ModConFlex, grant number 101073558) and by the Deutsche Forschungsgemeinschaft (DFG, German Research Foundation) Project-ID 531152215 -- CRC 1701.
Bouchra Elghazi would like to thank the GIPSA-lab for hosting her during her research visit. ChatGPT-5 was only used to proofread the manuscript.

\printbibliography

@article{GorZwaMas:2005,
author = {Le Gorrec, Y. and Zwart, H. and Maschke, B.},
title = {Dirac structures and Boundary Control Systems associated with Skew-Symmetric Differential Operators},
journal = {SIAM J. Control Optim.},
volume = {44},
number = {5},
pages = {1864--1892},
year = {2005},
doi = {10.1137/040611677},
}

@article{SchmidZwart:21,
	author = {J. Schmid and H. Zwart},
	title = {Stabilization of port-{H}amiltonian systems by nonlinear boundary control in the presence of disturbances},
	DOI= {10.1051/cocv/2021051},
	journal = {ESAIM: COCV},
	year = {2021},
	volume = {27},
	pages = {53},
}

@article{RamirezZwartLeGorrec:17,
title = {Stabilization of infinite dimensional port-{H}amiltonian systems by nonlinear dynamic boundary control},
journal = {Automatica},
volume = {85},
pages = {61--69},
year = {2017},
doi = {10.1016/j.automatica.2017.07.045},
author = {H. Ramirez and H. Zwart and Y. {Le Gorrec}},
}

@book{pazy2012semigroups,
  title={Semigroups of linear operators and applications to partial differential equations},
  author={Pazy, A.},
  volume={44},
  year={1983},
  publisher={Springer},
  series={Appl. Math. Sci.},
  doi={10.1007/978-1-4612-5561-1},
}

@phdthesis{Aug:16,
    author = {Augner, B.},
    title ={Stabilisation of Infinite-dimensional Port-{H}amiltonian Systems via Dissipative Boundary Feedback} ,
    school = {University of Wuppertal},
    year = {2016},
}

@article{Aug:19,
author = {Augner, B.},
title = {Well-Posedness and Stability of Infinite-Dimensional Linear Port-{H}amiltonian Systems with Nonlinear Boundary Feedback},
journal = {SIAM J. Control Optim.},
volume = {57},
number = {3},
pages = {1818--1844},
year = {2019},
doi = {10.1137/15M1024901},
}

@article{AugJac:14,
author = {Augner, B. and Jacob, B.},
year = {2014},
month = {12},
pages = {207-229},
title = {Stability and stabilization of infinite-dimensional linear port-{H}amiltonian systems},
volume = {3},
journal = {Evol. Equ. Control Theory},
doi = {10.3934/eect.2014.3.207},
}

@InProceedings{Aug:20,
author="Augner, B.",
title="Well-posedness and stability for interconnection structures of port-{H}amiltonian type",
booktitle="Control Theory of Infinite-Dimensional Systems",
year="2020",
publisher="Springer International Publishing",
pages="1--52",
doi = {10.1007/978-3-030-35898-3_1},
}

@article{Mironchenko:24,
  author    = {A. Mironchenko},
  title     = {Well-posedness and properties of the flow for semilinear evolution equations},
  journal   = {Math. Control Signals Syst.},
  year      = {2024},
  volume    = {36},
  number    = {3},
  pages     = {483--523},
  doi       = {10.1007/s00498-023-00378-x},
}

@article{MironchenkoPrieur:20,
author = {A. Mironchenko and C. Prieur},
title = {Input-to-State Stability of Infinite-Dimensional Systems: Recent Results and Open Questions},
journal = {SIAM Review},
volume = {62},
number = {3},
pages = {529-614},
year = {2020},
doi = {10.1137/19M1291248},
}

@article{Fattorini:68,
author = {H. O. Fattorini},
title = {Boundary Control Systems},
journal = {SIAM J. Control},
volume = {6},
number = {3},
pages = {349--385},
year = {1968},
doi = {10.1137/0306025},
}

@article{Clarke:98,
title = {Asymptotic Stability and Smooth {L}yapunov Functions},
journal = {J. Differ. Equ.},
volume = {149},
number = {1},
pages = {69--114},
year = {1998},
doi = {10.1006/jdeq.1998.3476},
author = {F. H. Clarke and Yu. S. Ledyaev and R. J. Stern},
}

@book{Vorabie:03,
title = {$C_0$-Semigroups and Applications},
author = {I. I. Vrabie},
series = {North-Holland Mathematics Studies},
publisher = {North-Holland},
volume = {191},
year = {2003},
doi = {10.1016/S0304-0208(03)80043-9},
}

@article{TuscnakWeiss:14,
title = {Well-posed systems—The {LTI} case and beyond},
author = {M. Tucsnak and G. Weiss},
journal = {Automatica},
volume = {50},
number = {7},
pages = {1757-1779},
year = {2014},
doi = {10.1016/j.automatica.2014.04.016},
}

@article{Macchelli2004,
author = {A. Macchelli and C. Melchiorri},
title = {Control by Interconnection and Energy Shaping of the {T}imoshenko Beam},
journal = {Math. Comput. Model. Dyn. Syst.},
volume = {10},
number = {3-4},
pages = {231--251},
year = {2004},
publisher = {Taylor \& Francis},
doi = {10.1080/13873950412331335243},
}

@article{CrepeauPrieur2006,
     author = {Cr\'epeau, E. and Prieur, C.},
     title = {Control of a clamped-free beam by a piezoelectric actuator},
     journal = {ESAIM Control Optim. Calc. Var.},
     pages = {545--563},
     year = {2006},
     publisher = {EDP-Sciences},
     volume = {12},
     number = {3},
     doi = {10.1051/cocv:2006008},
}

@phdthesis{Villegas:07,
    author = {Villegas, J. A.},
    title = {A Port-{H}amiltonian Approach to Distributed Parameter Systems},
    school = {University of Twente},
    year = {2007},
    doi = {10.3990/1.9789036524896},
}

@article{SchaftCervera:07,
title = {Interconnection of port-{H}amiltonian systems and composition of {D}irac structures},
journal = {Automatica},
volume = {43},
number = {2},
pages = {212-225},
year = {2007},
author = {J. Cervera and A. J. van der Schaft and A. Baños},
doi = {10.1016/j.automatica.2006.08.014},
}

@article{SchJelt:14,
title = {Port-{H}amiltonian Systems Theory: An Introductory Overview},
year = {2014},
volume = {1},
doi = {10.1561/2600000002},
number = {2-3},
pages = {173-378},
author = {van der Schaft, A. J. and Jeltsema, D},
journal = {Found. Trends Syst. Control},
}

@article{Schaft:06,
author = {van der Schaft, A. J.},
year = {2006},
pages = {1339-1366},
title = {Port-{H}amiltonian systems: an introductory survey},
volume={3},
journal = {Proc. Int. Congr. Math.},
doi = {10.4171/022-3/65},
}

@article{SchaftMaschke:02,
author = {A.J. van der Schaft and B.M. Maschke},
title = {{H}amiltonian formulation of distributed-parameter systems with boundary energy flow},
journal = {J. Geom. Phys.},
volume = {42},
number = {1},
pages = {166-194},
year = {2002},
issn = {0393-0440},
doi = {10.1016/S0393-0440(01)00083-3},
}

@ARTICLE{RashardScjaftStram:20,
  author={Rashad, R. and Califano, F. and van der Schaft, A. J. and Stramigioli, S.},
  journal={IMA J. Math. Control Inf.}, 
  title={Twenty years of distributed port-{H}amiltonian systems: a literature review}, 
  year={2020},
  volume={37},
  number={1},
  pages={1400-1422},
  doi={10.1093/imamci/dnaa018}
}

@ARTICLE{PaunLeGorrecRamirez:21,
  author={Paunonen, L. and Le Gorrec, Y. and Ramirez, H.},
  journal={IEEE Trans. Autom. Control}, 
  title={A {L}yapunov Approach to Robust Regulation of Distributed Port–{H}amiltonian Systems}, 
  year={2021},
  volume={66},
  number={12},
  pages={6041-6048},
  doi = {10.1109/TAC.2021.3069679},
}

@ARTICLE{HumLassi:18,
  author={Humaloja, J. P. and Paunonen, L.},
  journal={IEEE Trans. Autom. Control}, 
  title={Robust Regulation of Infinite-Dimensional Port-{H}amiltonian Systems}, 
  year={2018},
  volume={63},
  number={5},
  pages={1480-1486},
  doi = {10.1109/TAC.2017.2748055},
}

@inproceedings{humaloja2016robust,
  title={Robust Regulation for port-{H}amiltonian Systems of even order},
  author={Humaloja, J. P. and Paunonen, L. and Pohjolainen, S.},
  booktitle={Int. Symp. Math. Theory Netw. Syst.},
  pages={152-156},
  year={2016},
}

@inproceedings{HumLassiPohj:16,
author={Humaloja, J. P. and Paunonen, L. and Pohjolainen, S.},
year = {2016},
booktitle = {2016 Eur. Control Conf. (ECC)},
pages = {2203-2208},
title = {Robust Regulation for first-order port-{H}amiltonian Systems},
doi = {10.1109/ECC.2016.7810618},
}

@article{LeGorrecToledoRamirezWu:23,
  author={Toledo-Zucco, J. and Wu, Y. and Ramirez, H. and Le Gorrec, Y.},
  journal={IEEE Control Syst. Lett.}, 
  title={Infinite-dimensional observers for high-order boundary-controlled port-{H}amiltonian systems}, 
  year={2023},
  volume={7},
  pages={1676-1681},
  doi = {10.1109/LCSYS.2023.3278252},
}

@article{ElghJacZw:25,
    title={Well-posedness of a class of infinite-dimensional port-{H}amiltonian systems with boundary control and observation}, 
    author={B. Elghazi and B. Jacob and H. Zwart},
    journal = {IFAC-{PapersOnLine}},
    volume = {59},
    number = {8},
    pages = {102-107},
    year = {2025},
    doi = {10.1016/j.ifacol.2025.08.074},
}

@article{sontag1989,
  author  = {E. D. Sontag},
  title   = {Smooth Stabilization Implies Coprime Factorization},
  journal = {IEEE Trans. Autom. Control},
  year    = {1989},
  volume  = {34},
  number  = {4},
  pages   = {435--443}
}

@article{Sontag1995,
author = {E. D. Sontag},
title = {On the Input-to-State Stability Property},
journal = {Eur. J. Control},
volume = {1},
number = {1},
pages = {24-36},
year = {1995},
doi = {10.1016/S0947-3580(95)70005-X},
}

@book{Mironchenko2023,
  author    = {A. Mironchenko},
  title     = {Input-to-State Stability: Theory and Applications},
  publisher = {Springer},
  year      = {2023},
  doi       = {10.1007/978-3-031-14674-9}
}

@book{Krstic2019,
  author    = {I. Karafyllis and M. Krstic},
  title     = {Input-to-State Stability for PDEs},
  publisher = {Springer},
  year      = {2019},
  doi       = {10.1007/978-3-319-91011-6}
}

@article{MattioniWuLeGorrZw:22,
author = {A. Mattioni and Y. Wu and Y. {Le Gorrec} and H. Zwart},
title = {Stabilization of a class of mixed {ODE}–{PDE} port-{H}amiltonian systems with strong dissipation feedback},
journal = {Automatica},
volume = {142},
pages = {110284},
year = {2022},
issn = {0005-1098},
doi = {10.1016/j.automatica.2022.110284},
}

@ARTICLE{VillegasZwLeGorrMasche:09,
  author={Villegas, J. and Zwart, H. and Le Gorrec, Y. and Maschke, B.},
  journal={IEEE Trans. Autom. Control}, 
  title={Exponential Stability of a Class of Boundary Control Systems}, 
  year={2009},
  volume={54},
  number={1},
  pages={142-147},
  doi={10.1109/TAC.2008.2007176},
}

@article{macchelli2016synthesis,
  title={On the synthesis of boundary control laws for distributed port-{H}amiltonian systems},
  author={Macchelli, A. and Le Gorrec, Y. and Ramirez, H. and Zwart, H.},
  journal={IEEE Trans. Autom. Control},
  volume={62},
  number={4},
  pages={1700--1713},
  year={2017},
  publisher={IEEE},
  doi={10.1109/TAC.2016.2595263},
}

@ARTICLE{RamirezLeGorrMachelliZW:14,
  author={Ramirez, H. and Le Gorrec, Y. and Macchelli, A. and Zwart, H.},
  journal={IEEE Trans. Autom. Control}, 
  title={Exponential Stabilization of Boundary Controlled Port-{H}amiltonian Systems With Dynamic Feedback}, 
  year={2014},
  volume={59},
  number={10},
  pages={2849-2855},
  doi={10.1109/TAC.2014.2315754},
}

@article{MironKarafKrstic:19,
author = {Mironchenko, A. and Karafyllis, I. and Krstic, M.},
title = {Monotonicity Methods for Input-to-State Stability of Nonlinear Parabolic {PDE}s with Boundary Disturbances},
journal = {SIAM J. Control Optim.},
volume = {57},
number = {1},
pages = {510-532},
year = {2019},
doi = {10.1137/17M1161877},
}

@article{MazencPrieur:2011,
author = {F. Mazenc and C. Prieur},
title = {Strict {L}yapunov functions for semilinear parabolic partial differential equations},
journal = {Math. Control Relat. Fields},
volume = {1},
number = {2},
pages = {231-250},
year = {2011},
doi = {10.3934/mcrf.2011.1.231},
}

\end{document}